\documentclass[preprint,12pt]{elsarticle}

\usepackage{a4,amsmath,amssymb,amsopn,amsthm,graphics,latexsym}

\newtheorem{thm}{Theorem}[section]
\newtheorem{lem}[thm]{Lemma}

\newtheorem{cor}[thm]{Corollary}
\newtheorem{rem}[thm]{Remark}

\def\be#1\ee{\begin{equation}#1\end{equation}}
\newcommand{\bea}{\begin{eqnarray}}
\newcommand{\eea}{\end{eqnarray}}
\newcommand{\beaa}{\begin{eqnarray*}}
\newcommand{\eeaa}{\end{eqnarray*}}
\newcommand{\bei}{\begin{itemize}}
\newcommand{\eei}{\end{itemize}}
\newcommand{\bee}{\begin{enumerate}}
\newcommand{\eee}{\end{enumerate}}

\def\P{{\mathbb{P}}}

\def\R{\mathbb{R}}

\def\E{\mathbb{E}}

\def\N{{\mathbb N}}

\newcommand{\eps}{\varepsilon}

\def\s{{\sigma}}

\def\FF{{\mathcal F}}
\journal{Statistics and Probability Letters}
\begin{document}
\begin{frontmatter}
\title{Small deviations of the determinants of random matrices with Gaussian entries\tnoteref{label1}}
\tnotetext[label1]{This work was
supported by RFBR grants N 11-01-00285, 11-01-12006-ofi-m-2011}
\author{Nadezhda V. Volodko}
 \ead{nvolodko@gmail.com}
\address{Sobolev Institute of Mathematics
of RAS, 4 Acad. Koptyug avenue, 630090 Novosibirsk Russia}



\begin{abstract}
The probability of small deviations of the determinant of the matrix $AA^T$  is
estimated, where $A$ is an $n\times\infty$ random matrix with centered entries
having joint Gaussian distribution. The inequality obtained is sharp in a sense.
\end{abstract}


\end{frontmatter}




\section{Introduction and main results.}

We discuss the problem of estimating probabilities of small deviations for the 
determinants of random matrices of a special type. The need for the result of this 
kind emerged during obtaining the asymptotic expansion for the distributions of 
canonical $V$-statistics of the third order (Borisov and Volodko). 
Moreover, the problem itself seems to be of interest. Regarding the topic, there are 
papers 
devoted to the small deviations for the smallest singular values of random matrices 
(see Adamczak et al. (2012) and references therein). As for results for determinants, Li and Weil (2008) obtained the distribution of the determinant for 
the i.i.d. Gaussian case. In the present work, we consider more general 
object, but the result of Theorem 1.1 and Remark 1.2 below is not sharp for the case 
of i.i.d. Gaussian entries (see Remark 1.4). 

Let $((\tau_{ij})_{i,j=1}^\infty$ be an array of centered jointly Gaussian random variables such 
that
\be \label{condvar}
  \inf_{a_{ij}} \E\left(\tau_{kk}-\sum_{\min(i,j)<k} a_{ij}\tau_{ij}\right)^2 =1.
\ee

Let 
\[
  det_n:= det \left(\tau_{ij}\right)_{i,j=1}^n.
\]
The result is as follows.

\begin{thm} Under assumption  \eqref{condvar} for any $n\in \N$ and any $\eps>0$ we have
\[
   \P( |det_n|\le \eps)\le  \P\left( \prod_{j=1}^n |X_j|\le \eps\right)
\]
where $X_j$ are i.i.d. $N(0,1)$-distributed random variables. 
\end{thm}

\begin{rem} Inequality becomes equality on diagonal matrices with independent entries. Moreover, for the fixed $n$ and $\varepsilon\rightarrow 0$,
\begin{equation}\label{product}
 \P\left( \prod_{j=1}^n |X_j|\le \eps\right)\sim\Big(\frac{2}
{\sqrt{2\pi}}\Big)^n\eps\frac{|\log\eps|^{n-1}}{(n-1)!}.
\end{equation}
\end{rem}

\begin{cor}
Let $A=\{\tau_{ij}\}$ be an $n\times\infty$ random matrix  with centered entries
having joint Gaussian distribution.
Suppose that for each $k\leq n$, 
$$d_k:=\inf_{\{\alpha_{ij}\}}\mathbb{E}\Big(\tau_{kk}-\sum_{\min(i,j)<k}\alpha_{ij}
\tau_{ij}\Big)^2>0.$$
Then, for the fixed $n$ and $\varepsilon=\varepsilon(n)>0$ small enough,
\begin{equation}\label{log}
{\bf P}(\sqrt{\det AA^T}<\varepsilon)\leq 
{\bf P}(\prod_{j=1}^n |X_j|\leq \varepsilon_0 )\sim \Big(\frac{2}
{\sqrt{2\pi}}\Big)^n\varepsilon_0
\frac{|\log (\varepsilon_0)|^k}{k!},
\end{equation}
where $X_j$ are i.i.d. $N(0,1)$-distributed random variables,
$$\varepsilon_0=\frac{\varepsilon}{\prod_{i=1}^{n}|d_i^{1/2}|}.$$
\end{cor}
\begin{rem}
For Gaussian matrices with independent entries estimate $(\ref{log})$ is not
sharp. According {\rm Proposition 4.2} from {\rm \cite{WLi}}, if $M$ is a random matrix with i.i.d. 
standard complex Gaussian entries, then
\begin{equation}\label{LiWei}
\det MM*\sim\prod_{j=1}^n\chi_j^2.
\end{equation}
It is not difficult to show that the density in zero of the product on the right hand is bounded.
\end{rem}
\section{Proofs.}

\subsection{Proof of Theorem 1.1.} We use induction in $n$. Our main argument is a 
trivial one dimensional version
of Anderson inequality: for any centered Gaussian random variable $Y$,
 for any $r\in \R, \eps>0$ we have
\be \label{Anderson}
    \P(|Y+r|\le \eps)\le  \P(|Y|\le \eps ). 
\ee
Let us introduce a $\s$-algebra,
\[
   \FF_n:=\s\{ \tau_{ij},  \min(i,j)<k  \}.
\]
It follows from the definition of the determinant that
\[
   det_n =det_{n-1} \tau_{nn} + V_n,
\]
where $V_n$ is an  $\FF_n$-measurable random variable.
On the other hand, by \eqref{condvar} we can write
\[
  \tau_{nn} = X_n + W_n
\]
where $X_n$ is an $N(0,1)$-distributed random variable independent of  $\FF_n$
and $W_n:=\E(X_n|\FF_n)$ is an $\FF_n$-measurable (also normal) random variable. 
It follows that
\be \label{induction_det}
   det_n =det_{n-1} (X_n+W_n) + V_n := det_{n-1} X_n + V'_n,
\ee
where again $V'_n$ is an  $\FF_n$-measurable random variable. 

Now the induction goes as follows
\begin{eqnarray*}
   \P( |det_n|\le \eps) &=&  \E \ \P( |det_n|\le \eps|\FF_n)
 \\
 &=& \E \ \P( |det_{n-1} X_n + V'_n|\le \eps|\FF_n)  
 \\  
 &\le&  \E \ \P( |det_{n-1} X_n|\le \eps|\FF_n)
 \\
  &=&  \P( |det_{n-1} X_n|\le \eps )
\\
 &=& \E \ \P( |det_{n-1} X_n|\le \eps | X_n)
\\
 &=& \E \ \P \left( |det_{n-1}| \le \frac{\eps}{|X_n|}\  \big| X_n\right)
\\
 &\le& \E\ \P\left(   \prod_{j=1}^{n-1} |X_j| \le \frac{\eps}{|X_n|} \ \big| X_n\right)
\\
&\le& \E\ \P\left(   \prod_{j=1}^{n} |X_j| \le \eps  | X_n \right)
\\
&=& \P\left(   \prod_{j=1}^{n} |X_j| \le \eps\right),
\end{eqnarray*}
and we are done.

Here the first equality is the total probability formula for conditional probabilities or expectations.
The second equality comes from \eqref{induction_det}. In the third line we use that $X_n$ is independent of $\FF_n$ 
while  $V'_n, det_{n-1} $ are   $\FF_n$-measurable. Therefore, conditionnally on  $\FF_n$ the value $V'_n$ is the constant
while  $det_{n-1} X_n$ is a normal $N(det_{n-1}^2,0)$-distributed random variable
and we may apply to it inequality \eqref{Anderson}.
In the fourth line we return to unconditional probability by  the same total probability formula as in line one.
In the fifth line we re-condition again, this time with respect to $X_n$. The sixth line is trivial. In the seventh line we use that $det_{n-1}$
and  $X_n$ are independent, hence conditional distribution of $det_{n-1}$ is the same as the inconditional one.
Therefore, we may use inductional assumption. The remaining lines are trivial.
Theorem 1.1 is proved.

\subsection{Proof of Remark 1.2.}

$$\P\Big(|\prod_{j=1}^nX_j|<\varepsilon\Big)=\P\Big(\sum_{j=1}^n\log|X_j|<\log\varepsilon\Big).$$
Here $X_j$ are $N(0,1)$-distributed random variables.
Prove (\ref{product}) by induction in $n$. Denote 
$$S_n=\sum_{j=1}^n\log|X_j|.$$ Write down the density of $\log |X_1|$:
$$f_{\log|X_1|}(u)=\frac{2}{\sqrt{2\pi}}\exp\{-e^{2u}/2+u\}.$$
Below suppose that $t<0$ and $|t|$ is large enough.
\begin{equation}\label{int}
\P(S_n<t)=\int_{-\infty}^{\infty}\P(S_{n-1}<t-u)f_{\log|X_n|}(u)du.
\end{equation}
First, obtain the upper estimate of (\ref{int}).
$$\int_{-\infty}^{\infty}=\int_{-\infty}^{t+\log\log|t|}+
\int_{t+\log\log|t|}^{\log|t|}+\int_{\log|t|}^{\infty}.$$
$$\int_{-\infty}^{t+\log\log|t|}\leq
\frac{2}{\sqrt{2\pi}}\int_{-\infty}^{t+\log\log|t|}e^udu=
\frac{2}{\sqrt{2\pi}}e^t\log|t|=o\Big(e^t\frac{|t|^{n-1}}{(n-1)!}\Big).$$
Here $\P(S_{n-1}<t-u)$ and $\exp\{-e^{2u}/2\}$ are estimated by $1$.
$$\int_{t+\log\log|t|}^{\log|t|}\leq\Big(\frac{2}{\sqrt{2\pi}}\Big)^ne^t
\int_{t+\log\log|t|}^{\log|t|}\frac{(u-t)^{n-2}}{(n-2)!}(1+o(1))du$$ $$=
\Big(\frac{2}{\sqrt{2\pi}}\Big)^ne^t\frac{|t|^{n-1}}{(n-1)!}(1+o(1)).$$
In the second integral we used induction assumption.
$$\int_{\log|t|}^{\infty}\sim\Big(\frac{2}{\sqrt{2\pi}}\Big)^ne^t
\int_{\log|t|}^{\infty}\frac{(u-t)^{n-2}}{(n-2)!}\exp\{-e^{2u}/2\}du$$
$$=\Big(\frac{2}{\sqrt{2\pi}}\Big)^ne^t|t|^{n-2}
\int_{\log|t|}^{\infty}\frac{(u/|t|+1)^{n-2}}{(n-2)!}\exp
\Big\lbrace-\frac{1}{2}e^{2(u-\log|t|/2)+\log|t|}\Big\rbrace du$$
$$\leq\Big(\frac{2}{\sqrt{2\pi}}\Big)^ne^{3t/2}\frac{|t|^{n-2}}{(n-2)!}
\int_0^{\infty}(u+1)^{n-2}\exp\Big\lbrace-\frac{1}{2}e^{2u}\Big\rbrace 
du=o\Big(e^t\frac{|t|^{n-1}}{(n-1)!}\Big).$$
Here we used induction assumption and the trivial fact, that the product of
two numbers $\exp\{2u-\log|t|\}$ and $|t|$ is greater than their sum.

Then find the lower estimate of (\ref{int}):
$$\int_{-\infty}^{\infty}\P(S_{n-1}<t-u)f_{\log|X_n|}(u)du$$ $$\geq
\Big(\frac{2}{\sqrt{2\pi}}\Big)^ne^t\int_{t+\log|t|}^{-\log|t|}
\frac{(u-t)^{n-2}}{(n-2)!}(1+o(1))\exp\{-e^{2u}/2\}du$$
$$\geq\Big(\frac{2}{\sqrt{2\pi}}\Big)^ne^t\frac{|t|^{n-1}}{(n-1)!}(1+o(1)).$$
Remark 1.2 is proved.
\subsection{Proof of corollary 1.3.}
\begin{lem}
Let $A$ be an $n\times m$-matrix, $n\leq m$. Matrix $B$ is obtained by adding
a column $\{a_1,...,a_n\}^T$ on the right of $A$. Then
$$\det AA^T\leq\det BB^T.$$
\end{lem}
\begin{proof}
If the rows of matrix $A$ are linearly dependent then $\det AA^T=0\leq\det BB^T.$ Suppose that the rows are linearly independent.

Denote the rows of matrix $A$ by $A_1,...,A_n$. Then
$$AA^T=\{\langle A_i,A_j\rangle\}_{i,j\leq n};\ \ BB^T=\{\langle A_i,A_j\rangle +
a_ia_j\}_{i,j\leq n}.$$
\begin{eqnarray*}
\det BB^T &=&\begin{vmatrix}
\langle A_1,A_1\rangle & ... & \langle A_1,A_n \rangle  \\  
\langle A_1,A_2\rangle+a_1a_2 & ... & \langle A_2,A_n\rangle+a_2a_n \\  
... & ... & ... \\  
\langle A_1,A_n\rangle+a_1a_n & ... & \langle A_n,A_n\rangle+a_n^2 \\ 
\end{vmatrix}\\ &+&
\begin{vmatrix}
a_1^2 & ... & a_1a_n  \\  
\langle A_1,A_2\rangle+a_1a_2 & ... & \langle A_2,A_n\rangle+a_2a_n \\  
... & ... & ... \\  
\langle A_1,A_n\rangle+a_1a_n & ... & \langle A_n,A_n\rangle+a_n^2 \\ 
\end{vmatrix}
\\ &=& \begin{vmatrix}
\langle A_1,A_1\rangle & ... & \langle A_1,A_n \rangle  \\  
\langle A_1,A_2\rangle+a_1a_2 & ... & \langle A_2,A_n\rangle+a_2a_n \\  
... & ... & ... \\  
\langle A_1,A_n\rangle+a_1a_n & ... & \langle A_n,A_n\rangle+a_n^2 \\ 
\end{vmatrix}\\ &+&
\begin{vmatrix}
a_1^2 & ... & a_1a_n  \\  
\langle A_1,A_2\rangle & ... & \langle A_2,A_n\rangle \\  
... & ... & ... \\  
\langle A_1,A_n\rangle & ... & \langle A_n,A_n\rangle \\ 
\end{vmatrix} 
\\ &=& \det AA^T + \begin{vmatrix}
a_1^2 & ... & a_1a_n  \\  
\langle A_1,A_2\rangle & ... & \langle A_2,A_n\rangle \\  
... & ... & ... \\  
\langle A_1,A_n\rangle & ... & \langle A_n,A_n\rangle \\ 
\end{vmatrix}\\ &+&
\begin{vmatrix}
\langle A_1,A_1\rangle & ... & \langle A_1,A_n \rangle  \\  
a_1a_2 & ... & a_2a_n \\  
... & ... & ... \\  
\langle A_1,A_n\rangle & ... & \langle A_n,A_n\rangle \\ 
\end{vmatrix}+...+
\begin{vmatrix}
\langle A_1,A_1\rangle & ... & \langle A_1,A_n \rangle  \\  
\langle A_1,A_2\rangle & ... & \langle A_2,A_n \rangle \\  
... & ... & ... \\  
a_1a_n & ... & a_n^2 \\ 
\end{vmatrix}\\ &=& \det AA^T+\sum_{i,j\leq n}(-1)^{i+j}a_ia_jM_{ij},
\end{eqnarray*}
where $M_{ij}$ is a complementary minor to the element of matrix $AA^T$ with
coordinates $(i,j)$.
The last equality comes from the row expansions of $n$ determinants. 
Then, the matrix of algebraic complements $\{(-1)^{i+j}M_{ij}\}$ of the
positively definite matrix $AA^T$ is also positively definite, because
$$\{(-1)^{i+j}M_{ij}\}=\det AA^T(AA^T)^{-1},$$
where $(AA^T)^{-1}$ is an inverse matrix of $AA^T$.
Then,
$$\sum_{i,j\leq n}
(-1)^{i+j}a_ia_jM_{ij}=
\{a_1,...,a_n\}\times\{(-1)^{i+j}M_{ij}\}\times\{a_1,...,a_n\}^T\geq 0.$$
This fact finishes the proof of Corollary 1.3.
\end{proof}
{\bf Proof} of Remark 1.4 is analogous to the proof of Remark~1.2.

{\bf Acknowledgements.} The author is grateful to M.~A.~Lifshits for useful 
discussions.

\end{document}